\documentclass{article}

\usepackage{amsmath}
\usepackage{amssymb}
\usepackage{graphicx}
\usepackage{bm}
\usepackage{subfigure}
\usepackage{mathtools}

\usepackage{tikz}
\tikzset{
     block/.style={rectangle, draw, fill=green!10, text width=10em,
                   text centered, rounded corners, minimum height=3em},
     arrow/.style={-{Stealth[]}}
     }
\usetikzlibrary{positioning,arrows.meta}

\usepackage{mathdots}
\usepackage{booktabs}
\usepackage{rotating}
\usepackage[ruled,linesnumbered]{algorithm2e}
\usepackage[a4paper,left=2.8cm,right=2.8cm,top=2.5cm,bottom=2.5cm]{geometry}
\usepackage{fancyhdr}
\usepackage[framemethod=tikz]{mdframed}
\newmdtheoremenv[%
backgroundcolor=green!10,%
outerlinecolor=black,%
leftmargin=0,%
rightmargin=0,
innertopmargin =3pt,%
innerleftmargin = 5pt,
innerrightmargin = 5pt,
splittopskip = \topskip,%
skipabove = \baselineskip,%
skipbelow = \baselineskip,%
roundcorner=5, ntheorem]
{theorem}{Theorem}[section]

\newtheorem{assumption}{Assumption}

\newtheorem{lemma}{Lemma}[section]

\makeatletter 
\@addtoreset{equation}{section}
\makeatother  

\newtheorem{remark}{Remark}[section]

\newenvironment{proof}{{\noindent\it Proof.}\quad}{\hfill $\square$\\}

\begin{document}
\title{Is hyperinterpolation efficient in the approximation of singular and oscillatory functions?}

\author{Congpei An\footnotemark[1]
       \quad\quad Hao-Ning Wu\footnotemark[2]}

\renewcommand{\thefootnote}{\fnsymbol{footnote}}
\footnotetext[1]{School of Mathematics, Southwestern University of Finance and Economics, Chengdu, China (ancp@swufe.edu.cn).}
\footnotetext[2]{Department of Mathematics, The University of Hong Kong, Hong Kong, China (hnwu@connect.hku.hk).}


\maketitle

\begin{abstract}
Singular and oscillatory functions feature in numerous applications. The high-accuracy approximation of such functions shall greatly help us develop high-order methods for solving applied mathematics problems. This paper demonstrates that hyperinterpolation, a discrete projection method with coefficients obtained by evaluating the $L^2$ orthogonal projection coefficients using some numerical integration methods, may be inefficient for approximating singular and oscillatory functions. A relatively large amount of numerical integration points are necessary for satisfactory accuracy. Moreover, in the spirit of product-integration, we propose an efficient modification of hyperinterpolation for such approximation. The proposed approximation scheme, called efficient hyperinterpolation, achieves satisfactory accuracy with fewer numerical integration points than the original scheme. The implementation of the new approximation scheme is relatively easy. Theorems are also given to explain the outperformance of efficient hyperinterpolation over the original scheme in such approximation, with the functions assumed to belong to $L^1(\Omega)$, $L^2(\Omega)$, and $\mathcal{C}(\Omega)$ spaces, respectively. These theorems, as well as numerical experiments on the interval and the sphere, show that efficient hyperinterpolation has better accuracy in such approximation than the original one when the amount of numerical integration points is limited.
\end{abstract}



\section{Introduction}

Let $\Omega$ be a bounded region of $\mathbb{R}^s$, either the closure of a connected open domain or a smooth closed lower-dimensional manifold in $\mathbb{R}^s$. The region is assumed to have finite measure with respect to a given measure $\text{d}\omega$, that is,
\begin{equation*}
\int_{\Omega}\text{d}\omega=V<\infty.
\end{equation*}

We are interested in the efficient numerical approximation of functions in the form of
\begin{equation}\label{equ:F}
F(x)=K(x)f(x)
\end{equation}
by some polynomials on $\Omega$, where $K\in L^1(\Omega)$ is a real- or complex-valued absolutely integrable function, which needs not be continuous or of one sign, and $f\in\mathcal{C}(\Omega)$ is a continuous (and preferably smooth) function. By \emph{efficient}, we mean that a considerably small amount of sampling points is enough for such approximation with satisfactory accuracy. We also investigate scenarios of $K\in L^2(\Omega)$ and $\mathcal{C}(\Omega)$ to refine the general (but rough) analysis for the case of $K\in L^1(\Omega)$.

\subsection{Sources of functions on the form $F=Kf$}
Functions in the form of \eqref{equ:F} frequently feature in mathematical physics and applied mathematics \cite{MR3971246,MR1922922,MR856705,MR1283940}. For example, the fundamental solutions of the Helmholtz equation are given by

\begin{equation*}
G(x,y)=\begin{dcases}
\frac{i}{4} H_0^{(1)}(\kappa|x-y|)&\text{for }x,y\in\mathbb{R}^2\\
 \frac{1}{4\pi}\frac{e^{i\kappa|x-y|}}{|x-y|}&\text{for }x,y\in\mathbb{R}^3,
\end{dcases}
\end{equation*}
where $|x-y|$ denotes the usual Euclidean distance between $x$ and $y$, $ H_0^{(1)}(z)$ is the Hankel function of the first kind and of order zero, and $\kappa$ is known as the wave number when the equation is applied to waves. The fundamental solution of the biharmonic differential equation in $\mathbb{R}^2$ is given by
\begin{equation*}
G(x,y)=\frac{1}{8\pi}|x-y|^2\log{|x-y|}\quad\text{for }x,y\in\mathbb{R}^2.
\end{equation*}

Another important source of singular and oscillatory functions can be found in the study of
\begin{equation*}
\frac{Y_{\ell, k}(y)}{|x-y|}\quad  \text{for }x,y \in \mathbb{S}^2
\end{equation*}
for the electromagnetic field and wave computation \cite{MR2934227,chien,MR1922922}, where $\mathbb{S}^2:=\{(x,y,z)\in\mathbb{R}^3:x^2+y^2+z^2=1\}$, and $Y_{\ell,k}$ is the spherical harmonic of degree $\ell$ and order $k$.

As we can see, many fundamental solutions are functions with singularity and oscillatory behaviors. The approximation of such functions helps us develop approximation methods to solve related mathematical physics problems. Thus, designing an efficient method for such approximation is a fascinating area of computational mathematics.

\subsection{The approximation basics}
Let the space $L^p(\Omega)$ be equipped with the usual $L^p$ norm for $1\leq p\leq \infty$, that is, for $g\in L^p(\Omega)$,
\begin{equation*}
\|g\|_p:=\begin{cases}
\left(\int_{\Omega}|g|^p\text{d}\omega\right)^{1/p},& 1\leq p<\infty,\\
\sup_{x\in\Omega} |g(x)|, & p = \infty.
\end{cases}
\end{equation*}
The space $\mathcal{C}(\Omega)$ of continuous functions is also equipped with the $L^{\infty}$ norm. In particular, $L^p(\Omega)$ is a Hilbert space when $p=2$, with the $L^2$ inner product defined as $\langle v,z\rangle = \int_{\Omega}v\bar{z}\text{d}\omega$. This inner product also induces the $L^2$ norm, that is, $\|g\|_2=\sqrt{\langle g,g\rangle}$ for $g\in L^2(\Omega)$.

Let $\mathbb{P}_n$ be the linear space of polynomials on $\Omega$ of degree at most $n$, equipped with the $L^2$ inner product, and let $\{p_1,p_2\ldots,p_{d_n}\}\subset\mathbb{P}_n$ be an orthonormal basis of $\mathbb{P}_n$ in the sense of $\langle p_{\ell},p_{\ell'}\rangle = \delta_{\ell\ell'}$ for $1\leq \ell,\ell'\leq d_n$, where $d_n=\dim\mathbb{P}_n$. A typical constructive approximation scheme of degree $n$ for $F$ consists of two stages: evaluating the integrals
\begin{equation}\label{equ:projcoeff}
\int_{\Omega} Kfp_{\ell}\text{d}\omega,\quad \ell=1,2,\ldots,d_n,
\end{equation}
and then approximating $F$ by
\begin{equation}\label{equ:proj}
\mathcal{P}_nF:=\sum_{\ell=1}^{d_n}\left(\int_{\Omega} Kfp_{\ell}\text{d}\omega\right)p_{\ell}.
\end{equation}
This scheme \eqref{equ:proj} is the famous $L^2$ \emph{orthogonal projection} of $F$ onto $\mathbb{P}_n$. To link the orthogonal projection to applications immediately, a discrete approximation of the scheme \eqref{equ:proj}, now known as the \emph{hyperinterpolation}, was introduced by Sloan in 1995 \cite{sloan1995polynomial}. Let $\langle\cdot,\cdot\rangle_m$ be an $m$-point quadrature rule of the form
\begin{equation}\label{equ:quad}
\sum_{j=1}^mw_jg(x_j)\approx \int_{\Omega}g\text{d}\omega,
\end{equation}
where the quadrature points $x_j\in\Omega$ and weights $w_j>0$ for $j=1,2,\ldots,m$. With the assumption that the quadrature rule \eqref{equ:quad} has exactness degree $2n$, i.e.,
\begin{equation*}
\sum_{j=1}^mw_jg(x_j) = \int_{\Omega}g\text{d}\omega\quad \forall g\in\mathbb{P}_{2n},
\end{equation*}
the hyperinterpolant of degree $n$, constructed for the approximation of $F\in\mathcal{C}(\Omega)$, is defined as
\begin{equation}\label{equ:hyperinterpolation}
\mathcal{L}_nF:=\sum_{\ell=1}^{d_n}\langle Kf,p_{\ell}\rangle_mp_{\ell}.
\end{equation}
We refer the reader to \cite{MR1922922,MR2274179,le2001uniform,MR1761902,MR1619076,zbMATH01421286,MR1744380,MR1845243} for some follow-up works on the general analysis of hyperinterpolation and \cite{an2021lasso,MR4226998,montufar2021distributed,sloan2012filtered} for some variants of classical hyperinterpolation.

However, it is well known that if $K$ is singular and highly oscillatory, it is inefficient to evaluate the integrals \eqref{equ:projcoeff} directly using some classical numerical integration rules. Instead, one shall evaluate them in a semi-analytical way: for the evaluation of an integral of the form $\int_{\Omega} K(x)f(x)\text{d}\omega(x)$,
one shall replace $f$ by its polynomial interpolant or approximant of degree $n$, expressed as $\sum_{\ell=1}^{d_n}c_{\ell}p_{\ell}$, and evaluate the integral by
\begin{equation*}
\int_{\Omega} K(x)f(x)\text{d}\omega(x)\approx \sum_{\ell=1}^{d_n}c_{\ell}\int_{\Omega} K(x)p_{\ell}(x)\text{d}\omega(x).
\end{equation*}
This idea for numerical integration may be referred to as the \emph{product-integration rule} in the classical literature \cite{MR494863,MR577405,MR650061}. This rule was initially designed on $[-1,1]$ for $K\in L^1[-1,1]$ and $f\in\mathcal{C}[-1,1]$, and it converges to the exact integral as the number of quadrature points approaches the infinity if $K\in L^p[-1,1]$ for some $p>1$ is additionally assumed. In the context of highly oscillatory integrals with an oscillatory $K\in \mathcal{C}(\Omega)$, this approach is known as the \emph{Filon-type method} \cite{zbMATH02573604,MR2211043}. In most of these references, $f$ is approximated by its interpolant, and it is generally assumed that the \emph{modified moments}
\begin{equation}\label{equ:modifiedmoments}
\int_{\Omega} K(x)p_{\ell}(x)\text{d}\omega(x),\quad \ell = 1,2,\ldots,d_n
\end{equation}
can be computed accurately by using special functions or efficiently by invoking some stable iterative procedures. Besides, $f$ may also be replaced by its approximant. For example, the idea of replacing $f$ with its hyperinterpolant has emerged in the first paper \cite{sloan1995polynomial} on hyperinterpolation. It may be better to replace $f$ with its hyperinterpolant rather than the interpolant. The $L^2$ operator norm of hyperinterpolation is bounded if the regional area/volume $V$ is finite \cite{sloan1995polynomial}, but there is no guarantee of the boundedness of the $L^2$ operator norm of polynomial interpolation over general regions; see a piece of evidence from \cite{MR1619076}.

In this spirit, we propose \emph{efficient hyperinterpolation}, a general scheme for approximating functions in the form of \eqref{equ:F}, provided that the modified moment \eqref{equ:modifiedmoments} can be readily obtained. We approximate $f$ by its hyperinterpolant and the resulting scheme is defined as
\begin{equation}\label{equ:efficienthyper}
\mathcal{S}_nF:=\sum_{\ell=1}^{d_n}\left(\int_{\Omega}K(\mathcal{L}_nf)p_{\ell}\text{d}\omega\right)p_{\ell}.
\end{equation}
Along with the classical hyperinterpolation \eqref{equ:hyperinterpolation}, this scheme can be regarded as another discrete approximation of the $L^2$ orthogonal projection \eqref{equ:proj}. The main theoretical results of this paper are the stability and error analysis for this scheme, and this scheme is verified to be efficient when the amount of quadrature points is considerably small.

Although singular and oscillatory integration was well studied in the classical literature, we found these studies were not widely linked to hyperinterpolation. Here is a possible explanation for this gap. The required quadrature exactness degree $2n$ for a hyperinterpolant of degree $n$ \emph{de facto} ensures a sufficient amount of numerical integration points when $n$ is relatively large. Thus, directly evaluating the integrals \eqref{equ:projcoeff} by the classical numerical integration methods may also lead to relatively satisfactory accuracy.

In a recent work \cite{an2022quadrature}, we discussed what if the required exactness $2n$ is relaxed to $n+n'$, where $0<n'\leq n$. This discussion provides a regime where efficient hyperinterpolation may perform much more accurately than classical hyperinterpolation. In particular, if $K$ is continuous, we show that for the classical hyperinterpolation of degree $n$, the approximation error is bounded as
\begin{equation*}
\|\mathcal{L}_nF-F\|_2\lesssim E_{n'}(Kf),
\end{equation*}
where $E_{n'}(g)=\inf_{\chi\in\mathbb{P}_{n'}}\|g-\chi\|_{\infty}$ for $g\in\mathcal{C}(\Omega)$; while for efficient hyperinterpolation of degree $n$, there holds
\begin{equation*}
\|\mathcal{S}_nF-F\|_2\lesssim E_{n'}(f)+E_n(K\chi^*),
\end{equation*}
where $\chi^*\in\mathbb{P}_{n'}$ is the best uniform approximation of $f$ in $\mathbb{P}_{n'}$, that is, $\|f-\chi^*\|_{\infty}=E_{n'}(f)$. Thus, the controlling term $E_{n'}(Kf)$ is considerably greater than $E_{n'}(f)$ and $E_n(K\chi^*)$ when $n'<n$, $f$ is smooth enough, and $K$ is awkward enough to be approximated by lower degree polynomials, asserting the outperformance of efficient hyperinterpolation in this scenario.

The rest of this paper is organized as follows. In the next section, we review some results of the classical hyperinterpolation and discuss some properties of the efficient modification. The implementation of efficient hyperinterpolation is elaborated in Section \ref{sec:implementation}. In Section \ref{sec:general}, we analyze the stability and the error bound for efficient hyperinterpolation when $K \in L^1(\Omega)$. This analysis is refined in Section \ref{sec:refined} with the assumptions that $K \in L^2(\Omega)$ and $K \in \mathcal{C}(\Omega)$. In particular, we discuss in Section \ref{sec:wrong} why the classical hyperinterpolation may be inefficient when approximating functions in the form of \eqref{equ:F}. In Section \ref{sec:examples}, we investigate efficient hyperinterpolation on the interval and the sphere, respectively, and give some numerical results.

\section{Hyperinterpolation and efficient hyperinterpolation}\label{sec:hyperandefficienthyper}
Hyperinterpolation \eqref{equ:hyperinterpolation} uses classical numerical integration methods to evaluate the $L^2$ orthogonal projection coefficients \eqref{equ:projcoeff}. However, the classical methods prove to be inefficient in the presence of a singularity or an oscillatory $K$. Thus, we propose efficient hyperinterpolation \eqref{equ:efficienthyper} to achieve satisfactory approximation accuracy by using a considerably small amount of quadrature points. In this section, we review some results of \eqref{equ:hyperinterpolation} and discuss some properties of \eqref{equ:efficienthyper}.

\subsection{Hyperinterpolation}
As introduced in the previous section, the original definition \eqref{equ:hyperinterpolation} of hyperinterpolants of degree $n$ requires an $m$-point quadrature rule \eqref{equ:quad} with polynomial exactness $2n$ \cite{sloan1995polynomial}, and this requirement on quadrature exactness has been relaxed to $n+n'$ with $0<n'\leq n$ recently in \cite{an2022quadrature}.

The definition \eqref{equ:hyperinterpolation} is also restricted to the approximation of continuous functions. Thus, if $K$ is additionally assumed to be continuous, then it was derived in \cite{sloan1995polynomial} that $\mathcal{L}_nF$ defined by \eqref{equ:hyperinterpolation} with quadrature exactness $2n$ shall satisfy
\begin{equation}\label{equ:originalstability}
\|\mathcal{L}_nF\|_2\leq V^{1/2}\|F\|_{\infty}
\end{equation}
and
\begin{equation}\label{equ:originalerror}
\|\mathcal{L}_nF-F\|_2\leq 2V^{1/2}E_n(F),
\end{equation}
where $E_n(g)=\inf_{\chi\in\mathbb{P}_n}\|g-\chi\|_{\infty}$ denotes the best uniform approximation error of $g\in\mathcal{C}(\Omega)$ by a polynomial in $\mathbb{P}_n$.

Let the quadrature rule \eqref{equ:quad} have exactness degree $n+n'$ with $0<n'\leq n$, and let it satisfy the \emph{Marcinkiewicz--Zygmund} property that there exists an $\eta\in[0,1)$ such that
\begin{equation}\label{equ:MZproperty}
\left|\sum_{j=1}^mw_j\chi(x_j)^2-\int_{\Omega}\chi^2\text{d}\omega\right|\leq \eta \int_{\Omega}\chi^2\text{d}\omega\quad \forall \chi\in\mathbb{P}_{n},
\end{equation}
and $\eta=0$ if $n'=n$. The property \eqref{equ:MZproperty} is referred to as the Marcinkiewicz--Zygmund property as it can be regarded as the Marcinkiewicz--Zygmund inequality \cite{filbir2011marcinkiewicz,Marcinkiewicz1937,mhaskar2001spherical,MR1863015} applied to polynomials of degree at most $2n$; see \cite{an2022quadrature} for more details. If the quadrature rule \eqref{equ:quad} with exactness degree $n+n'$ satisfies the Marcinkiewicz--Zygmund property \eqref{equ:MZproperty} with $\eta\in[0,1)$, then it was derived in \cite{an2022quadrature} that
\begin{equation}\label{equ:recentstability}
\|\mathcal{L}_nF\|_2\leq \frac{V^{1/2}}{\sqrt{1-\eta}}\|F\|_{\infty}
\end{equation}
and
\begin{equation}\label{equ:recenterror}
\|\mathcal{L}_nF-F\|_2\leq \left(\frac{1}{\sqrt{1-\eta}}+1\right)V^{1/2}E_{n'}(F).
\end{equation}

For the sake of generality, we have the following assumption for the rest of this paper.
\begin{assumption}\label{assumption}
The quadrature rule \eqref{equ:quad} has exactness degree $n+n'$ with $0<n'\leq n$, and it satisfies the Marcinkiewicz--Zygmund property \eqref{equ:MZproperty} with $\eta\in[0,1)$.
\end{assumption}

\subsection{Properties of efficient hyperinterpolation}

We then make a short discussion on the relations among $L^2$ orthogonal projection $\mathcal{P}_n$, hyperinterpolation $\mathcal{L}_n$, and efficient hyperinterpolation $\mathcal{S}_n$. Note that $\mathcal{P}_n\chi=\chi$ for all $\chi\in\mathbb{P}_n$, while for $\mathcal{L}_n$ with quadrature exactness $n+n'$ ($0<n'\leq n$), there only holds $\mathcal{L}_n\chi=\chi$ for all $\chi\in\mathbb{P}_{n'}$, see \cite{an2022quadrature}.

We have the following lemma on the relation between $\mathcal{S}_n$ and $\mathcal{P}_n$.
\begin{lemma}\label{lem:21}
Let $K\in L^1(\Omega)$. Then $\mathcal{S}_n(K\chi) = \mathcal{P}_n(K\chi)$ for all $\chi\in\mathbb{P}_{n'}$.
\end{lemma}
\begin{proof}
This is shown immediately from the fact that $\mathcal{L}_n\chi=\chi$ for all $\chi\in\mathbb{P}_{n'}$. The coefficients $\int_{\Omega}K(\mathcal{L}_n\chi)p_{\ell}\text{d}\omega$ of $\mathcal{S}_n(K\chi)$ are exactly the coefficients $\int_{\Omega}K\chi p_{\ell}\text{d}\omega$ of $\mathcal{P}_n(K\chi)$.
\end{proof}

We then discuss the relation between $\mathcal{S}_n$ and $\mathcal{L}_n$. We can see that if $K=1$, then  $\mathcal{S}_nF=\mathcal{L}_nF$. Indeed, the coefficients of $\mathcal{S}_nF$ in this case are
\begin{equation*}
\int_{\Omega}(\mathcal{L}_nf)p_{\ell}\text{d}\omega=\int_{\Omega}\left(\sum_{\ell'=1}^{d_n}\langle f,p_{\ell'}\rangle_mp_{\ell'}\right)p_{\ell}\text{d}\omega=\langle f,p_{\ell}\rangle_m,\quad \ell=1,2,\ldots,d_n,
\end{equation*}
by the orthonormality of $\{p_{\ell}\}$ under the $L^2$ inner product. If $K\neq 1$, we have the following lemma.
\begin{lemma}\label{lem:22}
Let $K\in L^2(\Omega)$. Then $\langle K\mathcal{L}_nf-\mathcal{S}_nF,\chi\rangle=0$ for all $\chi\in\mathbb{P}_n$.
\end{lemma}
\begin{proof}
For any basis polynomial $p_{\ell}\in\{p_{\ell}\}_{\ell=1}^{d_n}$ of $\mathbb{P}_n$, we have
$$\langle S_nF,p_{\ell}\rangle = \int_{\Omega}\left[\sum_{\ell'=1}^{d_n}\left(\int_{\Omega}K(\mathcal{L}_nf)p_{\ell'}\text{d}\omega\right)p_{\ell'}\right]p_{\ell}\text{d}\omega = \int_{\Omega}K(\mathcal{L}_nf)p_{\ell}\text{d}\omega = \langle K\mathcal{L}_nf,p_{\ell}\rangle. $$
Thus for any polynomial $\chi\in\mathbb{P}_n$ expressed by $\{p_{\ell}\}_{\ell=1}^{d_n}$, we also have
$\langle K\mathcal{L}_nf-\mathcal{S}_nF,\chi\rangle=0$.
\end{proof}

Lemma \ref{lem:22} suggests that $\mathcal{S}_nF$ is an orthogonal projection of $K\mathcal{L}_nf$ onto $\mathbb{P}_n$ as well as the following least squares property.
\begin{theorem}
Let $K\in L^2(\Omega)$. Then
\begin{equation*}
\langle K\mathcal{L}_nf-\mathcal{S}_nF,K\mathcal{L}_nf-\mathcal{S}_nF\rangle=\min_{\chi\in\mathbb{P}_n}\langle K\mathcal{L}_nf-\chi,K\mathcal{L}_nf-\chi\rangle.
\end{equation*}
\end{theorem}
\begin{proof}
For any $\chi\in\mathbb{P}_n$, we have $K\mathcal{L}_nf-\chi=K\mathcal{L}_nf-\mathcal{S}_nF+\mathcal{S}_nF-\chi$, and by Lemma \ref{lem:22}, we have $\langle K\mathcal{L}_nf-\mathcal{S}_nF,\mathcal{S}_nF-\chi\rangle = 0$. Thus, the Pythagorean theorem suggests
\begin{equation*}
\|K\mathcal{L}_nf-\mathcal{S}_nF\|_2^2+\|\mathcal{S}_nF-\chi\|_2^2 = \|K\mathcal{L}_nf-\chi\|_2^2,
\end{equation*}
which implies
\begin{equation*}
\|K\mathcal{L}_nf-\mathcal{S}_nF\|_2^2\leq \|K\mathcal{L}_nf-\chi\|_2^2\quad\text{for all } \chi\in\mathbb{P}_n
\end{equation*}
and
\begin{equation*}
\|K\mathcal{L}_nf-\mathcal{S}_nF\|_2^2= \|K\mathcal{L}_nf-\chi\|_2^2\quad\text{if }\chi=\mathcal{S}_nF.
\end{equation*}
Hence the theorem is proved.
\end{proof}


\section{Implementation of efficient hyperinterpolation}\label{sec:implementation}

To implement efficient hyperinterpolation, the key step is to evaluate the coefficients of efficient hyperinterpolation \eqref{equ:efficienthyper}. Note that for $\ell=1,2,\ldots,d_n$, each coefficient
\begin{equation*}
\begin{split}
\int_{\Omega}K(\mathcal{L}_nf)p_{\ell}\text{d}\omega
&=\int_{\Omega}K\left[\sum_{\ell'=1}^{d_n}\left(\sum_{j=1}^mw_jf(x_j)p_{\ell'}(x_j)\right)p_{\ell'}\right]p_{\ell}\text{d}\omega\\
&=\sum_{j=1}^mw_j\left(\sum_{\ell'=1}^{d_n}p_{\ell'}(x_j)\int_{\Omega}Kp_{\ell'}p_{\ell}\text{d}\omega\right)f(x_j)\\
&=\sum_{j=1}^mW_{j\ell}f(x_j),
\end{split}
\end{equation*}
where
\begin{equation*}
W_{j\ell}:=w_j\left(\sum_{\ell'=1}^{d_n}p_{\ell'}(x_j)\int_{\Omega}Kp_{\ell'}p_{\ell}\text{d}\omega\right),\quad j=1,2,\ldots,m.
\end{equation*}
Thus, the weights $\{W_{j\ell}\}$ can be computed analytically or stably if one can evaluate
\begin{equation}\label{equ:alphaellell}
\alpha_{\ell'\ell}:=\int_{\Omega}Kp_{\ell'}p_{\ell}\text{d}\omega,\quad 1\leq \ell',\ell\leq d_n
\end{equation}
analytically or stably. Note that $p_{\ell'}p_{\ell}$ is another polynomial of degree $n_1+n_2$, where $n_1:=\deg{p_{\ell'}}$ and $n_2:=\deg{p_{\ell}}$. Then it can be expanded as
\begin{equation*}
p_{\ell'}p_{\ell}=\sum_{r=1}^{d_{n_1+n_2}}c_rq_r,
\end{equation*}
where $\{q_r\}_{r=1}^{d_{2n}}$ is an orthonormal basis of $\mathbb{P}_{2n}$, which could be chosen from the same orthogonal family of $\{p_{\ell}\}$ or not, and the coefficients
\begin{equation}\label{equ:cr}
c_r:= \int_{\Omega} p_{\ell'}p_{\ell}q_r\text{d}\mu,\quad r=1,2,\ldots,d_{n_1+n_2}.
\end{equation}
In the expression \eqref{equ:cr}, $\text{d}\mu$ is the Lebesgue--Stieltjes measure associated with $\mu$. Sometimes we may have $\text{d}\mu(x)=\mu(x)\text{d}x$, and $\mu(x)$ is referred to as the weight function of the orthogonal family $\{q_r\}$.

As introduced in Introduction, it is generally assumed that the modified moments
\begin{equation}\label{equ:betar}
\beta_r:=\int_{\Omega}Kq_r\text{d}\omega
\end{equation}
can be computed by using special functions or invoking some stable iterative procedures. In the implementation of efficient hyperinterpolation, we assume that $\beta_r$ can be computed analytically or stably for $r=1,2,\ldots,d_{2n}$. Thus, the weights
\begin{equation}\label{equ:Wjl}
W_{j\ell}=w_j\sum_{\ell'=1}^{d_n}p_{\ell'}(x_j)\alpha_{\ell'\ell}=w_j\left[\sum_{\ell'=1}^{d_n}p_{\ell'}(x_j)\left(\sum_{r=1}^{d_{n_1+n_2}}c_r\beta_r\right)\right]
\end{equation}
can be computed analytically or stably for $j=1,2,\ldots,m$ and $\ell=1,2,\ldots,d_n$.

The above discussion suggests how to implement efficient hyperinterpolation \eqref{equ:efficienthyper} in the form of
\begin{equation}\label{equ:implementation}
\mathcal{S}_{n}F=\sum_{\ell=1}^{d_n}\left(\sum_{j=1}^mW_{j\ell}f(x_j)\right)p_{\ell}.
\end{equation}
Here is a pseudocode describing the whole procedure.










\begin{center}
\noindent\fbox{
    \parbox{.9\linewidth}{
    \textbf{Algorithm. Efficient hyperinterpolant \eqref{equ:efficienthyper} for the approximation of $F=Kf$}\vspace{0.15cm}

    \texttt{Compute the modified moments \eqref{equ:betar} for $r=1,2,\ldots,d_{2n}$, save as $\{\beta_r\}_{r=1}^{d_{2n}}$;}

    \texttt{for $\ell=1:d_n$}

    \texttt{\quad for $\ell'=1:d_n$}

    \texttt{\quad\quad for $r=1:d_{n_1+n_2}$}

    \texttt{\quad\quad\quad $c_r=\langle p_{\ell'}p_{\ell},q_r\rangle$;}

    \texttt{\quad\quad end}

    \texttt{\quad\quad $\alpha_{\ell'\ell}=\sum_{r=1}^{d_{n_1+n_2}}c_r\beta_r$;}

    \texttt{\quad end}

    \texttt{\quad for $j=1:m$}

    \texttt{\quad\quad $W_{j\ell}= w_j\sum_{\ell'=1}^{d_n}p_{\ell'}(x_j)\alpha_{\ell'\ell}$}

    \texttt{\quad end}

    \texttt{end}

    \texttt{$\mathcal{S}_{n}F=\sum_{\ell=1}^{d_n}\left(\sum_{j=1}^mW_{j\ell}f(x_j)\right)p_{\ell}$.}
}
}
\end{center}
It can be seen that this algorithm is easy to be implemented.


















\section{Exploratory estimate: $K\in L^1(\Omega)$}\label{sec:general}
We now analyze efficient hyperinterpolation for the approximation of $F=Kf$ when $K\in L^1(\Omega)$. This case is the most general one among $K\in L^1(\Omega)$, $L^2(\Omega)$, and $\mathcal{C}(\Omega)$, as there holds $\mathcal{C}(\Omega)\subset L^2(\Omega)\subset L^1(\Omega)$ for a bounded and closed subset $\Omega$ of $\mathbb{R}^s$. As $L^1(\Omega)$ does not carry any inner products, we can only give a general but rough analysis. What's more, since $F=Kf\in L^1(\Omega)$, we can only give an $L^1$ error analysis. We shall refine our analysis in the next section by assuming $K\in L^2(\Omega)$ and $\mathcal{C}(\Omega)$.

\begin{theorem}\label{thm1}
Given $K\in L^1(\Omega)$ and $f\in\mathcal{C}(\Omega)$, let $F=Kf$ and let $\mathcal{S}_nF$ be defined as \eqref{equ:efficienthyper}, where the $m$-point quadrature rule \eqref{equ:quad} fulfills the Assumption \ref{assumption}. Then
\begin{equation}\label{equ:L1stability}
\|\mathcal{S}_nF\|_2\leq \frac{V^{1/2} A_n}{\sqrt{1-\eta}}\|f\|_{\infty},
\end{equation}
where
$$A_n=\sqrt{\sum_{\ell=1}^{d_n}\sum_{\ell'=1}^{d_n}\alpha_{\ell'\ell}^2}$$
with $\alpha_{\ell'\ell}$ defined as \eqref{equ:alphaellell}, and
\begin{equation}\label{equ:L1error}
\|\mathcal{S}_nF-F\|_1\leq \left(\frac{VA_n}{\sqrt{1-\eta}}+\|K\|_1\right)E_{n'}(f)+\left(V^{1/2}\sum_{\ell=1}^{d_n}\|p_{\ell}\|_{\infty}+1\right)E_n^{(1)}(K\chi^*),
\end{equation}
where $E_n^{(1)}(g):=\inf_{\chi\in\mathbb{P}_n}\|g-\chi\|_1$, and $\chi^*\in\mathbb{P}_{n'}$ is the best uniform approximation of $f$ in $\mathbb{P}_{n'}$.
\end{theorem}

\begin{proof}
By Parseval's identity, we have
\begin{equation*}
\|\mathcal{S}_nF\|_2^2=\sum_{\ell=1}^{d_n}\left(\int_{\Omega}K(\mathcal{L}_nf)p_{\ell}\text{d}\omega\right)^2=
\sum_{\ell=1}^{d_n}\left(\sum_{\ell'=1}^{d_n}\langle f,p_{\ell'}\rangle_m\int_{\Omega}Kp_{\ell'}p_{\ell}\text{d}\omega\right)^2.
\end{equation*}
By applying the Cauchy--Schwarz inequality and Parseval's identity again, we have
\begin{equation*}
\|\mathcal{S}_nF\|_2^2\leq \sum_{\ell=1}^{d_n}\left(\sum_{\ell'=1}^{d_n}\langle f,p_{\ell'}\rangle_m^2\right)\left(\sum_{\ell'=1}^{d_n}\alpha_{\ell'\ell}^2\right)=\|\mathcal{L}_nf\|_2^2\sum_{\ell=1}^{d_n}\sum_{\ell'=1}^{d_n}\alpha_{\ell'\ell}^2,
\end{equation*}
which leads to $\|\mathcal{S}_nF\|_2\leq A_n\|\mathcal{L}_nf\|_2$. By the stability result \eqref{equ:recentstability} with $F$ changed to $f$, we have the stability result \eqref{equ:L1stability}. For any $\chi\in\mathbb{P}_{n'}$, we have
\begin{equation*}\begin{split}
\|\mathcal{S}_nF-F\|_1
&=     \|\mathcal{S}_n(F-K\chi)-(F-K\chi)+(\mathcal{S}_n(K\chi)-K\chi)\|_1\\
&\leq  V^{1/2}\|\mathcal{S}_n(F-K\chi)\|_2+\|F-K\chi\|_1+\|\mathcal{S}_n(K\chi)-K\chi\|_1\\
&\leq  \frac{VA_n}{\sqrt{1-\eta}}\|f-\chi\|_{\infty}+\|K\|_1\|f-\chi\|_{\infty}+\|\mathcal{S}_n(K\chi)-K\chi\|_1,
\end{split}\end{equation*}
where the last inequality is obtained by applying the stability result \eqref{equ:L1stability} and H\"{o}lder's inequality to $F-K\chi=K(f-\chi)$, respectively. As the above estimate applied to an arbitrary $\chi\in\mathbb{P}_{n'}$, letting $\chi=\chi^*$ gives
\begin{equation}\label{equ:L1intermediate}
\|\mathcal{S}_nF-F\|_1
\leq  \left(\frac{VA}{\sqrt{1-\eta}}+\|K\|_1\right)E_{n'}(f)+\|\mathcal{S}_n(K\chi^*)-K\chi^*\|_1.
\end{equation}
By Lemma \ref{lem:21}, the term $\|\mathcal{S}_n(K\chi^*)-K\chi^*\|_1=\|\mathcal{P}_n(K\chi^*)-K\chi^*\|_1$. Thus for any $\chi\in\mathbb{P}_n$, we have $\mathcal{P}_n(K\chi^*)-K\chi^* = \mathcal{P}_n(K\chi^*-\chi)-(K\chi^*-\chi)$ and
\begin{equation*}\begin{split}
\|\mathcal{P}_n(K\chi^*)-K\chi^*\|_1 \leq V^{1/2} \|\mathcal{P}_n(K\chi^*-\chi)\|_2+\|K\chi^*-\chi\|_1.
\end{split}\end{equation*}
As for any $g\in L^1(\Omega)$, there holds
\begin{equation*}
\|\mathcal{P}_ng\|_2
=\left(\sum_{\ell=1}^{d_n}\left(\int_{\Omega}gp_{\ell}\text{d}\omega\right)^2\right)^{1/2}\leq \sum_{\ell=1}^{d_n}\left|\int_{\Omega}gp_{\ell}\text{d}\omega\right|\leq \sum_{\ell=1}^{d_n}\|gp_{\ell}\|_1\leq \|g\|_1\sum_{\ell=1}^{d_n}\|p_{\ell}\|_{\infty},
\end{equation*}
we have
\begin{equation*}
\|\mathcal{P}_n(K\chi^*)-K\chi^*\|_1\leq \left(V^{1/2}\sum_{\ell=1}^{d_n}\|p_{\ell}\|_{\infty}+1\right)\|K\chi^*-\chi\|_1.
\end{equation*}
Since this estimate applied to an arbitrary $\chi\in\mathbb{P}_n$, we have
\begin{equation*}
\|\mathcal{P}_n(K\chi^*)-K\chi^*\|_1\leq \left(V^{1/2}\sum_{\ell=1}^{d_n}\|p_{\ell}\|_{\infty}+1\right)E_n^{(1)}(K\chi^*).
\end{equation*}
Together with \eqref{equ:L1intermediate}, we have the error bound \eqref{equ:L1error}.
\end{proof}

\section{Refined estimates: $K\in L^2(\Omega)$ and $\mathcal{C}(\Omega)$}\label{sec:refined}

We then refine our general analysis in Section \ref{sec:general} by assuming $K\in L^2(\Omega)$ and $\mathcal{C}(\Omega)$, respectively. Inner products emerge as a powerful tool in such refinement. For example, we used the estimate $\|\mathcal{P}_ng\|_2\leq \|g\|_1\sum_{\ell=1}^{d_n}\|p_{\ell}\|_{\infty}$ for $g\in L^1(\Omega)$ in the proof of Theorem \ref{thm1}, but we have
\begin{equation}\label{equ:Pg2}
\|\mathcal{P}_ng\|_2\leq\|g\|_2\quad\forall g\in L^2(\Omega),
\end{equation}
and
\begin{equation}\label{equ:Pgc}
\|\mathcal{P}_ng\|_2\leq V^{1/2}\|g\|_{\infty}\quad\forall g\in \mathcal{C}(\Omega)
\end{equation}
with the aid of inner products. Indeed, according to the construction \eqref{equ:proj} of $\mathcal{P}_n$, we have $\langle \mathcal{P}_ng-g,\chi \rangle=0$ for all $\chi\in\mathbb{P}_n$. Thus, letting $\chi=\mathcal{P}_ng$, the Cauchy--Schwarz inequality implies
\begin{equation*}
\|\mathcal{P}_ng\|_2^2=\langle g,\mathcal{P}_ng\rangle\leq \|g\|_2\|\mathcal{P}_ng\|_2,
\end{equation*}
hence $\|\mathcal{P}_ng\|_2\leq \|g\|_2$ for $g\in L^2(\Omega)$. By generalized H\"{o}lder's inequality\footnote{There are many generalizations of the classical H\"{o}lder's inequality. In this paper, by \emph{generalized H\"{o}lder's inequality}, we mean the following one \cite[p. 186]{MR767633}. For $r,p_1,p_2,\ldots,p_n\in (0,\infty]$ and $f_k\in L^{p_k}(\Omega)$, $k=1,2,\ldots,n$, if
\begin{equation*}
\sum_{k=1}^n\frac{1}{p_k}=\frac{1}{r},
\end{equation*}
where $1/\infty$ is interpreted as $0$, then $\prod_{k=1}^nf_k\in L^r(\Omega)$ and there holds
\begin{equation*}
\left\|\prod_{k=1}^nf_k\right\|_r\leq \prod_{k=1}^n\|f_k\|_{p_k}.
\end{equation*}
}, $\|g\|_2\leq V^{1/2}\|g\|_{\infty}$ for $g\in\mathcal{C}(\Omega)$, thus $\|\mathcal{P}_ng\|_2\leq V^{1/2}\|g\|_{\infty}$ for $g\in\mathcal{C}(\Omega)$.

\subsection{Analysis with $K\in L^2(\Omega)$}
When $K\in L^2(\Omega)$, we have the following theorem.
\begin{theorem}\label{thm2}
Let $K\in L^1(\Omega)\bigcap L^2(\Omega)$ and adopt the rest conditions of Theorem \ref{thm1}. Then
\begin{equation}\label{equ:L2stability}
\|\mathcal{S}_nF\|_2\leq\|K\|_2\|\mathcal{L}_n\|_{\infty}\|f\|_{\infty},
\end{equation}
where $\|\mathcal{L}_n\|_{\infty}$ denotes the norm of $\mathcal{L}_n$ as an operator from $\mathcal{C}(\Omega)$ to $\mathcal{C}(\Omega)$, and
\begin{equation}\label{equ:L2error}
\|\mathcal{S}_nF-F\|_2\leq(1+\|\mathcal{L}_n\|_{\infty}) \|K\|_2E_{n'}(f)+2E_n^{(2)}(K\chi^*),
\end{equation}
where $E_n^{(2)}(g):=\inf_{\chi\in\mathbb{P}_n}\|g-\chi\|_2$, and $\chi^*\in\mathbb{P}_{n'}$ is the best uniform approximation of $f$ in $\mathbb{P}_{n'}$.
\end{theorem}

\begin{proof}
As $\mathcal{S}_nF\in\mathbb{P}_n$, Lemma \ref{lem:22} suggests $\langle \mathcal{S}_nF,\mathcal{S}_nF\rangle = \langle K\mathcal{L}_nf,\mathcal{S}_nF\rangle$. Then, by the Cauchy--Schwarz inequality and generalized H\"{o}lder's inequality,
\begin{equation*}
\|\mathcal{S}_nF\|_2^2\leq\|K\mathcal{L}_nf\|_2\|\mathcal{S}_nF\|_2\leq \|K\|_2\|\mathcal{L}_nf\|_{\infty}\|\mathcal{S}_nF\|_2,
\end{equation*}
which leads to $\|\mathcal{S}_nF\|_2\leq\|K\|_2\|\mathcal{L}_n\|_{\infty}\|f\|_{\infty}$. For any $\chi\in\mathbb{P}_{n'}$, we have
\begin{equation*}\begin{split}
\|\mathcal{S}_nF-F\|_2
&=     \|\mathcal{S}_n(F-K\chi)-(F-K\chi)+(\mathcal{S}_n(K\chi)-K\chi)\|_2\\
&\leq  \|\mathcal{S}_n(F-K\chi)\|_2+\|F-K\chi\|_2+\|\mathcal{S}_n(K\chi)-K\chi\|_2\\
&\leq  \|K\|_2 \|\mathcal{L}_n\|_{\infty} \|f-\chi\|_{\infty}+\|K\|_2\|f-\chi\|_{\infty}+\|\mathcal{S}_n(K\chi)-K\chi\|_2,
\end{split}\end{equation*}
where the last inequality is obtained by applying the stability \eqref{equ:L2stability} and generalized H\"{o}lder's inequality to $F-K\chi=K(f-\chi)$, respectively. Letting $\chi=\chi^*$ gives
\begin{equation}\label{equ:L2intermediate}
\|\mathcal{S}_nF-F\|_2
\leq  \left(\|\mathcal{L}\|_{\infty}+1\right)\|K\|_2E_{n'}(f)+\|\mathcal{S}_n(K\chi^*)-K\chi^*\|_2.
\end{equation}

Similar to the proof of Theorem \ref{thm1}, Lemma \ref{lem:21} implies $\|\mathcal{S}_n(K\chi^*)-K\chi^*\|_2=\|\mathcal{P}_n(K\chi^*)-K\chi^*\|_2$. By the estimate \eqref{equ:Pg2}, for any $\chi\in\mathbb{P}_n$, we have
\begin{equation*}\begin{split}
\|\mathcal{P}_n(K\chi^*)-K\chi^*\|_2\leq \|\mathcal{P}_n(K\chi^*-\chi)\|_2+\|K\chi^*-\chi\|_2\leq 2 \|K\chi^*-\chi\|_2.
\end{split}\end{equation*}
Since this estimate applied to an arbitrary $\chi\in\mathbb{P}_n$, we have $\|\mathcal{P}_n(K\chi^*)-K\chi^*\|_2\leq 2E_n^{(2)}(K\chi^*)$. Together with \eqref{equ:L2intermediate}, we have the error bound \eqref{equ:L2error}.
\end{proof}

\begin{remark}
For $\mathcal{L}_n$ with quadrature exactness $2n$, some studies on $\|\mathcal{L}_n\|_{\infty}$ in various regions can be found in \cite{caliari2007hyperinterpolation,caliari2008hyperinterpolation,MR2491427,le2001uniform,zbMATH01421286,MR1744380,wang2014norm}. This operator norm awaits further investigation for $\mathcal{L}_n$ with quadrature exactness $n+n'$ ($0<n'< n$). Nevertheless, the norm $\|\mathcal{L}_n\|_{\infty}$ cannot be uniformly bounded by any constant in general.
\end{remark}

\begin{remark}
The fact that $\|\mathcal{L}_n\|_{\infty}$ is not uniformly bounded has spurred the development of \emph{filtered hyperinterpolation} on the sphere and then on general regions \cite{MR4226998,montufar2021distributed,sloan2012filtered}. The filtered hyperinterpolation operator, as an operator from $\mathcal{C}(\Omega)\rightarrow\mathcal{C}(\Omega)$, has a uniformly bounded norm. Thus, a possible future work may be the combination of efficient and filtered hyperinterpolation so that a better result of the approximation of $F=Kf$ with $K\in L^2(\Omega)$ can be expected.
\end{remark}

\subsection{Analysis with $K\in\mathcal{C}(\Omega)$}
If $K$ is continuous, then we have the following analysis.
\begin{theorem}\label{thm3}
Let $K\in L^1(\Omega)\bigcap L^2(\Omega)\bigcap\mathcal{C}(\Omega) $ and adopt the rest conditions of Theorem \ref{thm1}. Then
\begin{equation}\label{equ:contstability}
\|\mathcal{S}_nF\|_2\leq \frac{V^{1/2}}{\sqrt{1-\eta}}\|K\|_{\infty}\|f\|_{\infty},
\end{equation}
where $\|\mathcal{L}_n\|_{\infty}$ denotes the norm of $\mathcal{L}_n$ as an operator from $\mathcal{C}(\Omega)$ to $\mathcal{C}(\Omega)$, and
\begin{equation}\label{equ:conterror}
\|\mathcal{S}_nF-F\|_2\leq \left(\frac{V^{1/2}}{\sqrt{1-\eta}}\|K\|_{\infty}+\|K\|_2\right)E_{n'}(f)+2V^{1/2}E_n(K\chi^*),
\end{equation}
where $\chi^*\in\mathbb{P}_{n'}$ is the best uniform approximation of $f$ in $\mathbb{P}_{n'}$.
\end{theorem}

\begin{proof}
In the proof of Theorem \ref{thm2}, we have obtained $\|\mathcal{S}_nF\|_2\leq \|K\mathcal{L}_nf\|_{2}$. Thus, for $K\in\mathcal{C}(\Omega)$, by generalized H\"{o}lder's inequality, we have $\|\mathcal{S}_nF\|_2\leq \|K\|_{\infty}\|\mathcal{L}_nf\|_{2}$, and by the stability result \eqref{equ:recentstability} of $\mathcal{L}_n$, we have the stability \eqref{equ:contstability} of $S_n$.

Similar to the case of $K\in L^2(\Omega)$, for any $\chi\in\mathbb{P}_{n'}$, we have
\begin{equation*}\begin{split}
\|\mathcal{S}_nF-F\|_2
&\leq \|\mathcal{S}_n(F-K\chi)\|_2+\|K(f-\chi)\|_2+\|\mathcal{S}_n(K\chi)-K\chi\|_2 \\
&\leq \frac{V^{1/2}}{\sqrt{1-\eta}}\|K\|_{\infty}\|f-\chi\|_{\infty}+\|K\|_2\|f-\chi\|_{\infty}+\|\mathcal{S}_n(K\chi)-K\chi\|_2.
\end{split}\end{equation*}
Letting $\chi=\chi^*$ leads to
\begin{equation}\label{equ:contitermediate}
\|\mathcal{S}_nF-F\|_2=\left(\frac{V^{1/2}}{\sqrt{1-\eta}}\|K\|_{\infty}+\|K\|_2  \right)E_{n'}(f)+\|\mathcal{S}_n(K\chi^*)-K\chi^*\|_2.
\end{equation}

By Lemma \ref{lem:21} and the estimate \eqref{equ:Pgc}, for any $\chi\in\mathbb{P}_n$, we have
\begin{equation*}\begin{split}
\|\mathcal{S}_n(K\chi^*)-K\chi^*\|_2
&=\|\mathcal{P}_n(K\chi^*)-K\chi^*\|_2\leq \|\mathcal{P}_n(K\chi^*-\chi)\|_2 + \|K\chi^*-\chi\|_2\\
&\leq V^{1/2}\|K\chi^*-\chi\|_{\infty}+ V^{1/2}\|K\chi^*-\chi\|_{\infty}=2V^{1/2}\|K\chi^*-\chi\|_{\infty}.
\end{split}\end{equation*}
Thus $\|\mathcal{S}_n(K\chi^*)-K\chi^*\|_2\leq 2V^{1/2}E_n(K\chi^*)$. Together with \eqref{equ:contitermediate}, we have the estimate \eqref{equ:conterror}.
\end{proof}

\subsection{The potential inefficiency of classical hyperinterpolation for the approximation of $F=Kf$}\label{sec:wrong}
The classical hyperinterpolation \eqref{equ:hyperinterpolation} is defined to approximate continuous functions. The approximation of $F=Kf\in\mathcal{C}(\Omega)$ by efficient hyperinterpolation is described by Theorem \ref{thm3}. Thus, if we let $K=1$, then both the stability result \eqref{equ:contstability} and the error bound \eqref{equ:conterror} of efficient hyperinterpolation reduce to \eqref{equ:recentstability} and \eqref{equ:recenterror} of the classical hyperinterpolation, respectively, derived in \cite{an2022quadrature}. Furthermore, if the quadrature rule \eqref{equ:quad} has exactness degree $2n$, that is, $\eta=0$, then they reduce to the original results \eqref{equ:originalstability} and \eqref{equ:originalerror} derived by Sloan in \cite{sloan1995polynomial}.




But what if $K\neq 1$ and $K$ is awkward enough to be approximated? In this case,
\begin{equation*}
\|\mathcal{S}_nF-F\|_2\lesssim E_{n'}(f)+ E_n(K\chi^*),
\end{equation*}
where $\chi^*$ is the best uniform approximation of $f$ in $\mathbb{P}_{n'}$. However, for the classical hyperinterpolation there holds
\begin{equation*}
\|\mathcal{L}_nF-F\|_2\lesssim E_{n'}(Kf).
\end{equation*}
Thus, if $f$ is smooth enough so that $E_n(K\chi^*)$ dominates the bound of $\|\mathcal{S}_nF-F\|_2$, and if $n'<n$ and $K$ is awkward enough so that $E_{n'}(Kf)$ is considerably greater than $E_n(K\chi^*)$, efficient hyperinterpolation shall give a better approximation than the classical one in the sense of estimated error bounds.

On the other hand, it is inappropriate to claim that efficient hyperinterpolation is always better than the classical hyperinterpolation in the approximation of $F=Kf$. If the singularity of $K$ is relatively weak (for a singular $K$), or if $K$ oscillates slowly (for an oscillatory $K$), then the classical hyperinterpolation may generate a comparable or even better approximation of $F$ than efficient hyperinterpolation.

\section{Examples and numerical experiments}\label{sec:examples}
We now investigate efficient hyperinterpolation \eqref{equ:efficienthyper} on two specific regions, the interval $[-1,1]\subset\mathbb{R}$ and the unit sphere $\mathbb{S}^2\subset \mathbb{R}^3$. Numerical experiments are also conducted. On each region, we test oscillatory and singular terms $K$, respectively. A key issue is how to evaluate the modified moments \eqref{equ:betar} analytically or stably. We shall discuss the computational issues of the moments separately on each region and for each $K$. All numerical results are carried out by using MATLAB R2022a on a laptop (16 GB RAM, Intel CoreTM i7-9750H Processor) with macOS Monterey 12.3.

\subsection{On the interval}

Let $\Omega =[-1,1]$. In this case, $d_n=n+1$. There is merit in adopting orthogonal polynomials as the basis \cite{MR2061539,MR0000077} for the approximation of functions defined on $[-1,1]$. In our experiments, we let $\{p_{\ell}\}_{\ell=1}^{d_n}$ be normalized Legendre polynomials $\{\tilde{P}_{\ell}\}_{\ell=0}^{n}$, and let $\{q_{r}\}_{r=1}^{d_{2n}}$ be Chebyshev polynomials $\{T_r\}_{r=0}^{2n}$. Thus for any $\tilde{P}_{\ell'}\tilde{P}_{\ell}$ with $0\leq\ell',\ell\leq n$, it can be expressed as $\tilde{P}_{\ell'}\tilde{P}_{\ell} = \sum_{r=0}^{2n}c_rT_{r}$, where the coefficients are given for $r\geq1$ by
\begin{equation*}
c_r = \frac{2}{\pi}\int_{-1}^1\frac{\tilde{P}_{\ell'}(x)\tilde{P}_{\ell}(x)T_r(x)}{\sqrt{1-x^2}}\text{d}x,\quad r = 1,\ldots,2n,
\end{equation*}
and for $r=0$ by the same formula with the factor $\pi/2$ changed to $1/\pi$ for $r=0$ \cite{MR3012510}. In the expression of $c_r$, $(1-x^2)^{-0.5}$ is the weight function associated to the Chebyshev polynomials, and $\langle \tilde{P}_{\ell'}\tilde{P}_{\ell},T_r\rangle$ is divided by the factor $\langle T_r,T_r\rangle$ since $\{T_r\}_{r=0}^{2n}$ are not orthonormal. In our experiments, these coefficients $\{c_r\}$ are obtained by the \texttt{chebcoeffs} command included in Chebfun \cite{driscoll2014chebfun}. For the quadrature rule \eqref{equ:quad}, we use the Gauss--Legendre quadrature. It is well-known that the $m$-point Gauss--Legendre quadrature has exactness degree $2m-1$.

\textbf{Oscillatory functions.} We first test $K(x) = e^{i\kappa x}$ with $\kappa>0$, which is an oscillatory term regularly appearing in applications.
For the evaluation of
\begin{equation}\label{equ:1dhighmoments}
\beta_r=\int_{-1}^1e^{i\kappa x} T_{r}(x)\text{d}x,\quad r  = 0,1,\ldots,2n,
\end{equation}
we invoke the stable algorithm proposed in \cite{MR2846755} for implementing the Filon--Clenshaw--Curtis rule \cite{dominguez2013filon,MR2846755}. For the function $f\in\mathcal{C}[-1,1]$, we let $f(x) = (1.2-x^2)^{-1}$.

For $\kappa=100$, we let $n=120$ and $m=70$; that is, the theoretical error of classical hyperinterpolation is controlled by $E_{19}(e^{i100 x}f)$, while that of efficient hyperinterpolation is controlled by $E_{19}(f)$ and $E_{120}(e^{i100 x}\chi^*)$, where $\chi^*\in\mathbb{P}_{19}$ is the best uniform approximation of $f$ in $\mathbb{P}_{19}$. The approximation results are displayed in Figure \ref{fig:1d1}, in which we see that efficient hyperinterpolation generates a good approximation, but the classical one fails to do so. Moreover, for $\kappa=160$, we let $n=180$ and $m=100$; that is, the theoretical error of classical hyperinterpolation is controlled by $E_{19}(e^{i160 x}f)$, while that of efficient hyperinterpolation is controlled by $E_{19}(f)$ and $E_{180}(e^{i160 x}\chi^*)$, where $\chi^*\in\mathbb{P}_{19}$ is the best uniform approximation of $f$ in $\mathbb{P}_{19}$. The approximation results are displayed in Figure \ref{fig:1d2}, which convey the same message as the case of $\kappa=100$.

\begin{figure}[htbp]
  \centering
  \includegraphics[width=\textwidth]{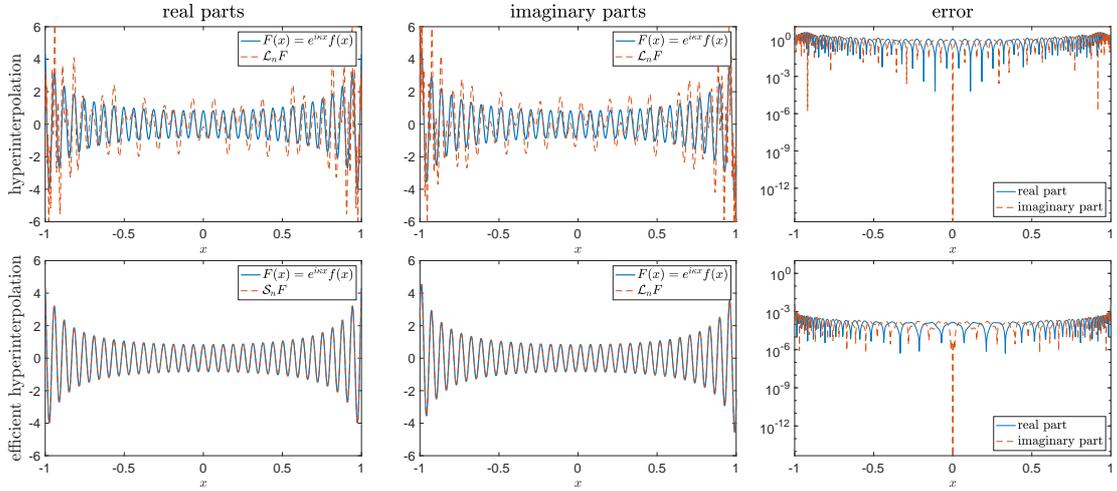}
  \caption{Approximation of $F(x)=e^{i\kappa x}(1.2-x^2)^{-1}$ by $\mathcal{L}_n$ and $\mathcal{S}_n$ with $(\kappa,n,m)=(100,120,70)$.}\label{fig:1d1}
\end{figure}

\begin{figure}[htbp]
  \centering
  \includegraphics[width=\textwidth]{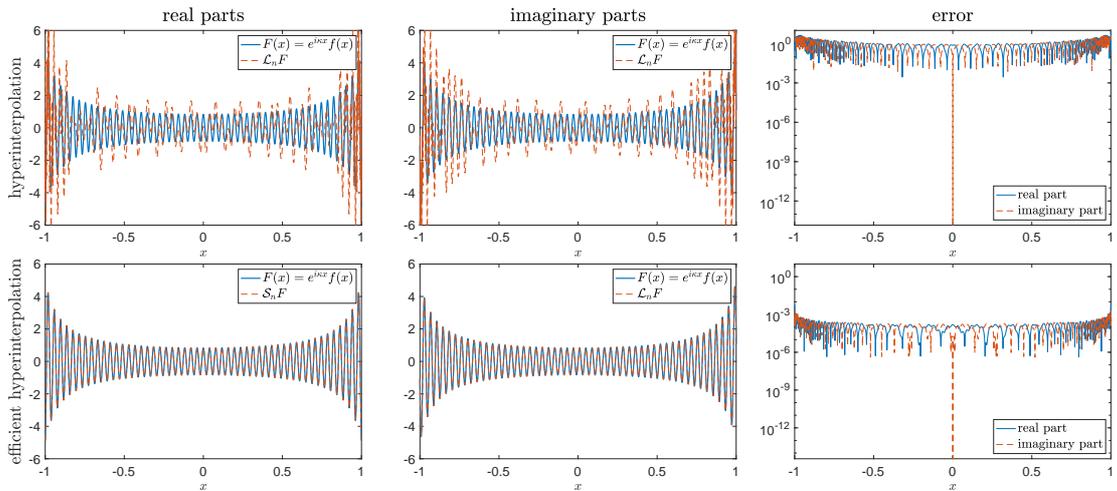}
  \caption{Approximation of $F(x)=e^{i\kappa x}(1.2-x^2)^{-1}$ by $\mathcal{L}_n$ and $\mathcal{S}_n$ with $(\kappa,n,m)=(160,180,100)$.}\label{fig:1d2}
\end{figure}

We continue with a more detailed investigation on the approximation of $F(x) =e^{i\kappa x}(1.2-x^2)^{-1} $ with $\kappa=100$ and 160. For $\kappa=100$, we test $n=100$, 120, and 150; for $\kappa=160$, we consider $n=160$, 180, and 210. For each $(\kappa,n)$, we test several numbers $m$ of quadrature points. The $L^2$ errors of each hyperinterpolant are listed in Table \ref{tab:HOF}. In each setting, the error of efficient hyperinterpolation is always less than that of classical hyperinterpolation. Apart from this, Table \ref{tab:HOF} conveys some other interesting messages. Let $n$ be fixed. When the exactness degree of the quadrature rule is less than $2n$, i.e., $2m-1<2n$, the limited number of quadrature points slow the convergence of classical hyperinterpolation, as its error bound is controlled by $E_{n'}(Kf)=E_{2m-1-n}(Kf)$. Meanwhile, efficient hyperinterpolation may work well because its error bound is controlled by $E_{n'}(f)=E_{2m-1-n}(f)$ and $E_n(Kf)$. When $2m-1\geq 2n$, by our analysis, the accuracy of both schemes only depends on $n$. On the other hand, let $m$ be fixed. When $2m-1<2n$, increasing $n$ may not help in improving the accuracy of classical hyperinterpolation; on the contrary, it may slow its convergence, as $E_{2m-1-n}(Kf)$ is enlarged as $n$ increases. However, if $E_n(Kf)$ dominates the error bound of efficient hyperinterpolation, then increasing $n$ shall improve the accuracy of efficient hyperinterpolation.

\begin{table}[htbp]
  \centering
  \footnotesize
  \setlength{\abovecaptionskip}{0pt}
\setlength{\belowcaptionskip}{10pt}
  \caption{Performance of hyperinterpolation and efficient hyperinterpolation with different $(n,m)$ for the approximation of $F(x)=K(x)f(x)$ with $K(x) = e^{i\kappa x}$ and $f(x) = (1.2-x^2)^{-1}$, with $\kappa=100$ and 160. Errors, measured by $\|\cdot\|$, are in the $L^2$ sense.}\label{tab:HOF}
  \ttfamily
    \begin{tabular}{|c||c|c|c|c|c|c|}
   \hline
    ~ & \multicolumn{2}{c|}{$n=100$} & \multicolumn{2}{c|}{$n=120$} & \multicolumn{2}{c|}{$n=150$}\\
    ~ & \multicolumn{2}{c|}{$K(x)=e^{i100x}$} & \multicolumn{2}{c|}{$K(x)=e^{i100x}$} & \multicolumn{2}{c|}{$K(x)=e^{i100x}$}\\ \hline
          $m$ & $\|\mathcal{L}_nF-F\|$ & $\|\mathcal{S}_nF-F\|$ & $\|\mathcal{L}_nF-F\|$ & $\|\mathcal{S}_nF-F\|$ & $\|\mathcal{L}_nF-F\|$ & $\|\mathcal{S}_nF-F\|$ \\ \hline
60  & 2.1437 & 0.2064 & 2.3310     & 2.1556     & 2.7565     & 2.6291 \\\hline
70  & 1.7667 & 0.2064 & 2.1339     & 3.7060e-04 & 2.3565     & 2.3635 \\\hline
80  & 1.3929 & 0.2064 & 1.7547     & 8.2733e-06 & 2.2603     & 0.02830 \\\hline
100 & 0.3428 & 0.2064 & 1.0354     & 8.2730e-06 & 1.5477     & 8.3481e-10 \\\hline
120 & 0.2064 & 0.2064 & 1.8091e-05 & 8.2730e-06 & 0.7998     & 1.4644e-13 \\\hline
150 & 0.2064 & 0.2064 & 8.2730e-06 & 8.2730e-06 & 9.6996e-14 & 9.2094e-14 \\\hline
180 & 0.2064 & 0.2064 & 8.2730e-06 & 8.2730e-06 & 7.6783e-14 & 6.9940e-14 \\\hline
  \end{tabular}
  \begin{tabular}{|c||c|c|c|c|c|c|}
   \hline
    ~ & \multicolumn{2}{c|}{$n=160$} & \multicolumn{2}{c|}{$n=180$} & \multicolumn{2}{c|}{$n=210$}\\
    ~ & \multicolumn{2}{c|}{$K(x)=e^{i160x}$} & \multicolumn{2}{c|}{$K(x)=e^{i160x}$} & \multicolumn{2}{c|}{$K(x)=e^{i160x}$}\\ \hline
       $m$ & $\|\mathcal{L}_nF-F\|$ & $\|\mathcal{S}_nF-F\|$ & $\|\mathcal{L}_nF-F\|$ & $\|\mathcal{S}_nF-F\|$ & $\|\mathcal{L}_nF-F\|$ & $\|\mathcal{S}_nF-F\|$ \\ \hline
70  & 2.6149 & 2.3755 & 2.8822     & 2.5556     & 3.2681     & 2.8106 \\\hline
100 & 2.0502 & 0.2014 & 2.2357     & 3.7455e-04 & 2.3368     & 2.3372 \\\hline
120 & 1.5749 & 0.2014 & 1.7994     & 5.8491e-05 & 2.1408     & 4.8505e-06 \\\hline
150 & 0.8957 & 0.2014 & 1.1128     & 5.8491e-05 & 1.4421     & 1.6188e-13 \\\hline
180 & 0.2014 & 0.2014 & 1.1543e-04 & 5.8491e-05 & 0.7253     & 1.4140e-13 \\\hline
210 & 0.2014 & 0.2014 & 5.8491e-05 & 5.8491e-05 & 2.9417e-13 & 2.0787e-13 \\\hline
240 & 0.2014 & 0.2014 & 5.8491e-05 & 5.8491e-05 & 1.2553e-13 & 1.2040e-13 \\\hline
  \end{tabular}
\end{table}

\textbf{Singular functions.} We then test three singular terms $K$, which are
\begin{equation*}
K(x)=
\begin{cases}
(1+x)^{-1/3},& \\
|x-1|^{-0.2},&\\
(1-x^2)^{-0.5}.&
\end{cases}
\end{equation*}
For the first two cases, we compute
\begin{equation*}
\beta_r=\int_{-1}^1K(x) T_{r}(x)\text{d}x,\quad r  = 0,1,\ldots,2n,
\end{equation*}
by the built-in command \texttt{quadgk} in MATLAB, which is a stable procedure developed in \cite{MR2379343}. For the third case, as $(1-x^2)^{-0.5}$ is the weight function associated to the Chebyshev polynomials, we have $\beta_0=\pi$ and $\beta_{r}=0$ for all $r\geq1$. For the continuous function $f\in\mathcal{C}[-1,1]$, we let $f(x) = e^{-x^2}$.

For each $K$, we report the $L^1$ errors of classical and efficient hyperinterpolation with $n=6,9,12,\ldots,120$, and $m=\lceil 1.1n/2\rceil$, $\lceil 1.2n/2\rceil$, and $\lceil 1.5n/2\rceil$. These errors are plotted in Figure \ref{fig:1dsingular}. We can summarize from these errors that when the available data (the number of quadrature points) is limited, then the error of efficient hyperinterpolation is generally less than that of classical hyperinterpolation. It is also interesting to see that classical hyperinterpolation may perform better than efficient hyperinterpolation as the amount of quadrature points increases. For example, see the subplots on the bottom left and bottom right of Figure \ref{fig:1dsingular}. An interesting related fact is that the functions $K(x)=(1+x)^{-1/3}$ and $K(x)=(1-x^2)^{-0.5}$ is smoother than $K(x)=|x-1|^{-0.2}$ in the sense of differentiability. Hence, it is interesting to identify the critical number of quadrature points that the outperformance of the classical and efficient hyperinterpolation switches as future work. In particular, this critical number may be related to the smoothness of $F$.

\begin{figure}[htbp]
  \centering
  \includegraphics[width=\textwidth]{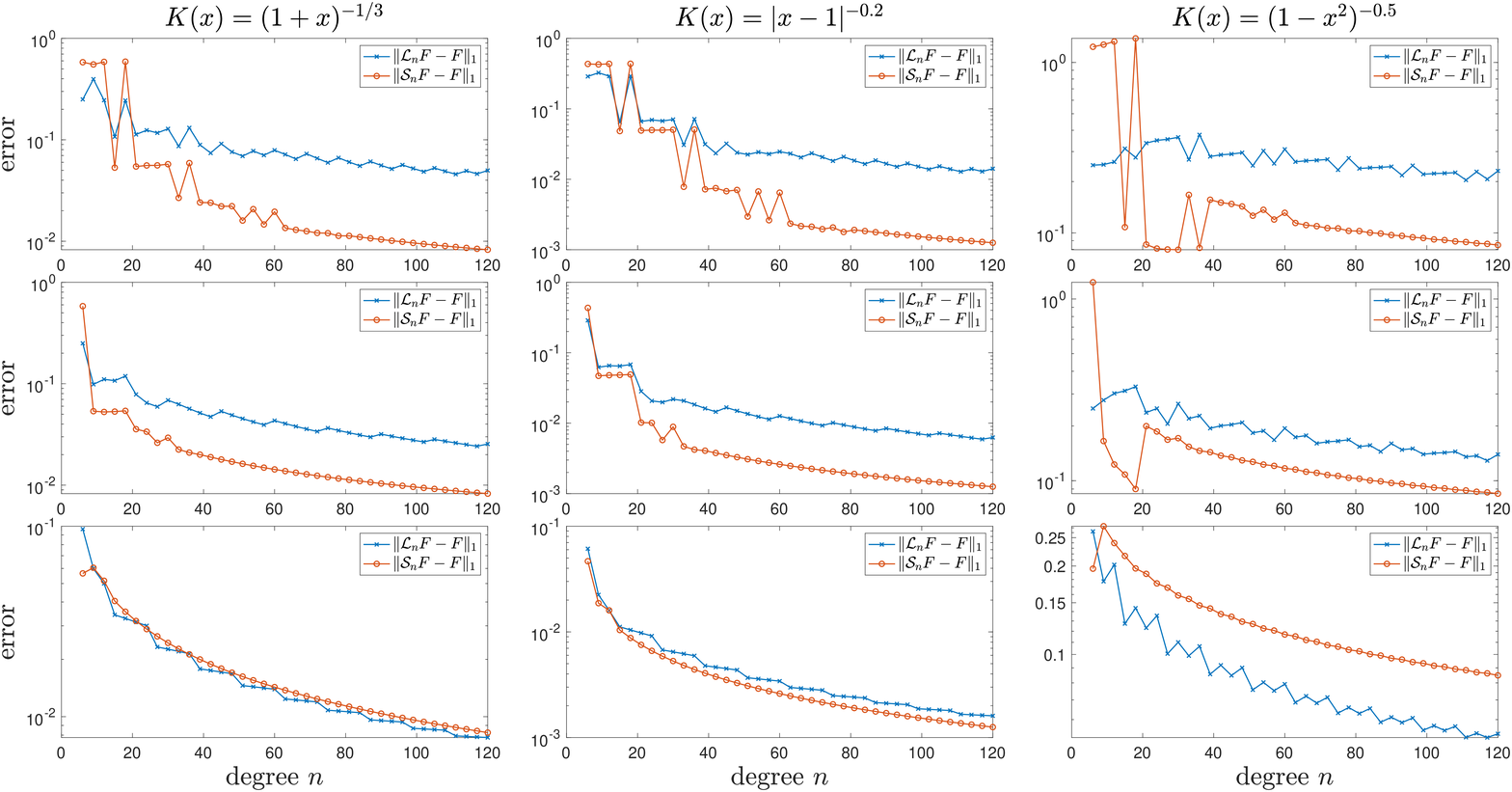}
  \caption{Performance of hyperinterpolation and efficient hyperinterpolation with different $(n,m)$ for the approximation of $F(x)=K(x)f(x)$ with three singular $K$'s and $f(x)=e^{-x^2}$. From top row to bottom row: $m=\lceil 1.1n/2\rceil$, $\lceil 1.2n/2\rceil$, and $\lceil 1.5n/2\rceil$, respectively.}\label{fig:1dsingular}
\end{figure}

\subsection{On the Sphere}

Let $\Omega = \mathbb{S}^2\subset\mathbb{R}^3$ with $\text{d}\omega(x)  = \omega (x)\text{d}x$, where $\omega(x)$ is an area measure on $\mathbb{S}^2$. Thus $V=\int_{\mathbb{S}^2}\text{d}\omega  = 4\pi$ denotes the surface area of $\mathbb{S}^2$. In this example, $\mathbb{P}_n$ can be regarded as the space of spherical polynomials of degree at most $n$. Let the basis $\{p_{\ell}\}_{\ell=1}^{d_n}$ be a set of orthonormal real spherical harmonics $ \{ Y_{\ell,k}: \ell  = 0, 1,\ldots, n,\text{ and }k = -\ell,-\ell+1,\ldots,\ell -1,\ell\}$, and the dimension of $\mathbb{P}_n$ is $d_n=\dim\mathbb{P}_n=(n+1)^2$. Let $\{q_r\}_{r=1}^{d_{2n}}$ also be the set of orthonormal real spherical harmonics $ \{ Y_{\ell,k}: \ell  = 0, 1,\ldots, 2n,\text{ and }k = -\ell,-\ell+1,\ldots,\ell -1,\ell\}$.

For the quadrature rule \eqref{equ:quad}, we use the rule based on spherical $t$-designs, which can be implemented easily and efficiently. A point set $\{x_1,x_2,\ldots,x_m\}\subset\mathbb{S}^2$ is said to be a \emph{spherical $t$-design} \cite{delsarte1991geometriae} if it satisfies
\begin{equation}\label{equ:sphericalquadrature}
\frac{1}{m}\sum_{j=1}^mv(x_j)=\frac{1}{4\pi}\int_{\mathbb{S}^2}v\text{d}\omega\quad\forall v\in\mathbb{P}_t.
\end{equation}
In other words, it is a set of points on the sphere such that an equal-weight quadrature rule in these points integrates all (spherical) polynomials up to degree $t$ exactly. Spherical $t$-designs require at least $(t+1)^2$ quadrature points to achieve the exactness degree $t$. For generating spherical $t$-designs, we make use of the well-conditioned spherical $t$-designs \cite{MR2763659} with $m = (t+1)^2$.

For any $Y_{\ell',k'}Y_{\ell,k}$ with $0\leq\ell',\ell\leq n$, $-\ell'\leq k'\leq \ell'$, and $-\ell\leq k\leq \ell$, it can be expressed as
\begin{equation*}
Y_{\ell',k'}Y_{\ell,k} = \sum_{\ell''=0}^{2n}\sum_{k''=-\ell''}^{\ell''}c_{\ell''k''}Y_{\ell'',k''},
\end{equation*}
where the coefficients
\begin{equation*}
c_{\ell''k''} =\int_{\mathbb{S}^2}(Y_{\ell',k'}Y_{\ell,k})Y_{\ell'',k''}\text{d}\omega,\quad \ell'' = 0,\ldots,2n,\quad k''=-\ell'',\ldots,\ell''
\end{equation*}
are evaluated by a quadrature rule using spherical $(\ell+\ell'+\ell'')$-designs.


We may use boldface letters to denote a point on $\mathbb{S}^2$, say ${\bm{x}}=[x,y,z]^{\text{T}}$, in order to avoid any potential ambiguity. The Euclidean distance between two points $\bm{\xi}$ and $\bm{x}$ on the sphere $\mathbb{S}^2$ is defined as $|\bm{\xi}-\bm{x}|:=\sqrt{2(1-\bm{\xi}\cdot \bm{x})}$, where ``$\cdot$'' denotes the inner product in $\mathbb{R}^3$.

\textbf{Oscillatory functions.} The spherical harmonics themselves are highly oscillatory when their degrees become relatively large. Thus we let $K=Y_{\bar{\ell},\bar{k}}$ for some $\bar{\ell},\bar{k}\in\mathbb{N}$. In this case, the modified moments can be evaluated by
\begin{equation*}
\beta_{r}:=\beta_{\ell''k''}=\int_{\mathbb{S}^2}Y_{\bar{\ell},\bar{k}}Y_{\ell'',k''}\text{d}\omega = \delta_{\bar{\ell},\ell''}\delta_{\bar{k},k''}.
\end{equation*}
For the continuous function $f\in\mathcal{C}(\mathbb{S}^2)$, we let $f(\bm{x})=f(x,y,z) =\cos(\cosh(xz) - 2y)$.

We investigate two kinds of oscillatory terms, $(\bar{\ell},\bar{k})=(12,8)$ and $(32,-24)$. For $K=Y_{12,8}$, we let $n=20$ and $m=625$, that is, $t=24$, the theoretical error of classical hyperinterpolation is controlled by $E_{4}(Y_{12,8}f)$, while that of efficient hyperinterpolation is controlled by $E_{4}(f)$ and $E_{20}(Y_{12,8}\chi^*)$, where $\chi^*\in\mathbb{P}_{4}$ is the best uniform approximation of $f$ in $\mathbb{P}_{4}$. The approximation results are displayed in the upper row of Figure \ref{fig:sphere}, in which we see that efficient hyperinterpolation generates a good approximation, but the classical one does not. For $K=Y_{32,-24}$, we let $n=40$ and $m=2209$, that is, $t=46$, the theoretical error of classical hyperinterpolation is controlled by $E_{6}(Y_{32,-24}f)$, while that of efficient hyperinterpolation is controlled by $E_{6}(f)$ and $E_{40}(Y_{32,-24}\chi^*)$, where $\chi^*\in\mathbb{P}_{6}$ is the best uniform approximation of $f$ in $\mathbb{P}_{6}$. The approximation results are displayed in the lower row of Figure \ref{fig:sphere}, which convey the same message as the case of $K=Y_{12,8}$.

\begin{figure}[htbp]
  \centering
  \begin{subfigure}
  \centering
  \includegraphics[width=\textwidth]{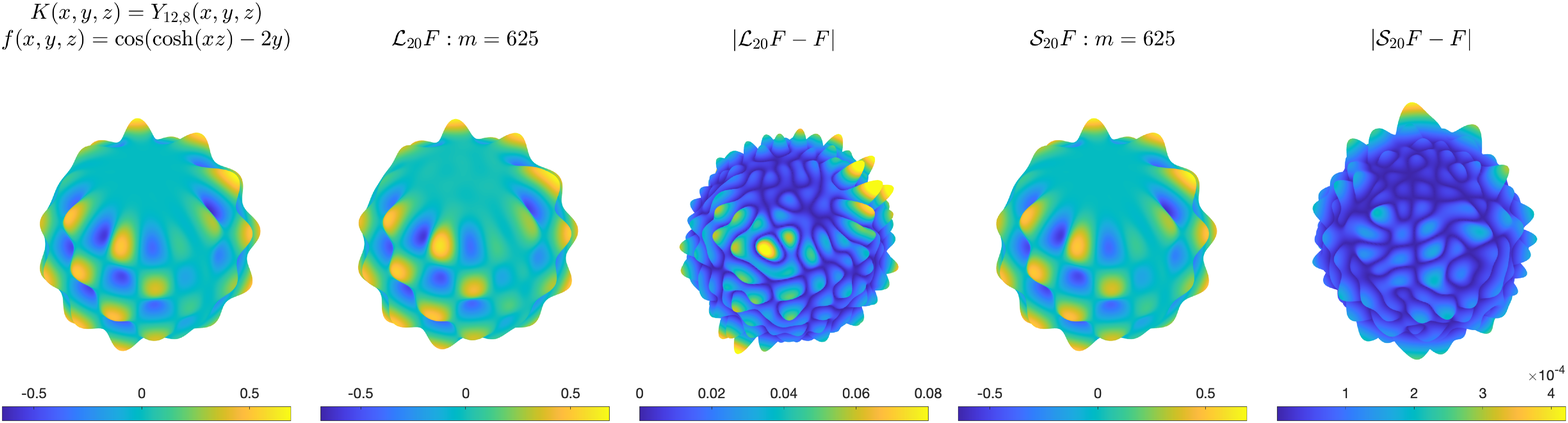}
\end{subfigure}\\
\begin{subfigure}
  \centering
  \includegraphics[width=\textwidth]{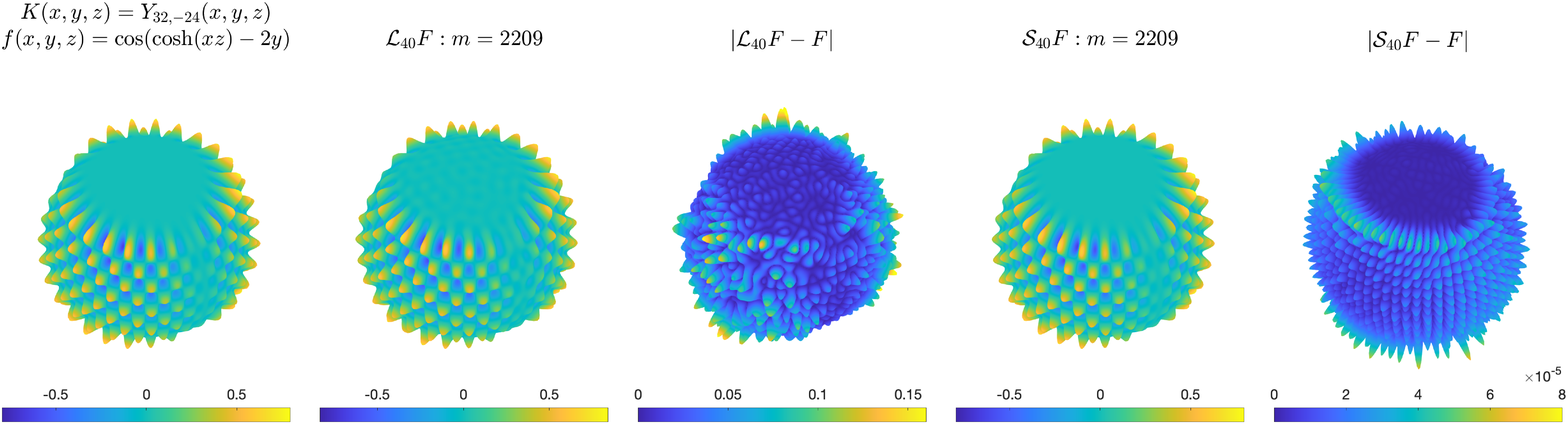}
\end{subfigure}
  \caption{Approximation of $F=Y_{12,8}f$ and $F=Y_{12,8}f$ with $f(x,y,z)=\cos(\cosh(xz) - 2y)$ by hyperinterpolation $\mathcal{L}_n$ and efficient hyperinterpolation $\mathcal{S}_n$.}\label{fig:sphere}
\end{figure}

Similar to Table \ref{tab:HOF}, we list the $L^2$ errors of the classical and efficient hyperinterpolation in different settings in Table \ref{tab:sphericalHOF}. We see that the error of efficient hyperinterpolation is always less than (or eventually equal to) that of the classical hyperinterpolation.

\begin{table}[htbp]
  \centering
    \footnotesize
\setlength{\abovecaptionskip}{0pt}
\setlength{\belowcaptionskip}{10pt}
  \caption{Performance of hyperinterpolation and efficient hyperinterpolation with different $(n,m)$ for the approximation of $F(x,y,z)=K(x,y,z)f(x,y,z)$ with two $K$'s and $f(x,y,z)=\cos(\cosh(xz) - 2y)$. Errors, measured by $\|\cdot\|$, are in the $L^2$ sense.}\label{tab:sphericalHOF}
  \ttfamily
  \begin{tabular}{|c||c|c|c|c|c|c|}
   \hline
    ~ & \multicolumn{2}{c|}{$n=16$} & \multicolumn{2}{c|}{$n=18$} & \multicolumn{2}{c|}{$n=20$}\\
    ~ & \multicolumn{2}{c|}{$K(\bm{x}) = Y_{12,8}(\bm{x})$} & \multicolumn{2}{c|}{$K(\bm{x}) = Y_{12,8}(\bm{x})$} & \multicolumn{2}{c|}{$K(\bm{x}) = Y_{12,8}(\bm{x})$}\\ \hline
          $m$ & $\|\mathcal{L}_nF-F\|$ & $\|\mathcal{S}_nF-F\|$ & $\|\mathcal{L}_nF-F\|$ & $\|\mathcal{S}_nF-F\|$ & $\|\mathcal{L}_nF-F\|$ & $\|\mathcal{S}_nF-F\|$ \\ \hline
484  & 0.1427 & 0.0116 & 0.1359     & 0.0092     & 0.1233 & 0.0082 \\\hline
529  & 0.1271 & 0.0097 & 0.1160     & 0.0031     & 0.1181 & 0.0044 \\\hline
576  & 0.1090 & 0.0086 & 0.0993     & 9.1533e-04 & 0.0932 & 8.0575e-04 \\\hline
625  & 0.0910 & 0.0086 & 0.0973     & 7.0753e-04 & 0.0861 & 2.9376e-04 \\\hline
841  & 0.0530 & 0.0086 & 0.0425     & 5.9738e-04 & 0.0439 & 5.9812e-05 \\\hline
1089 & 0.0285 & 0.0086 & 0.0189     & 5.9737e-04 & 0.0112 & 5.9767e-05 \\\hline
1369 & 0.0098 & 0.0086 & 6.4698e-04 & 5.9737e-04 & 1.6743e-04 & 5.9767e-05 \\\hline
1681 & 0.0086 & 0.0086 & 5.9749e-04 & 5.9737e-04 & 5.9776e-05 & 5.9767e-05 \\\hline
2025 & 0.0086 & 0.0086 & 5.9737e-04 & 5.9737e-04 & 5.9767e-05 & 5.9767e-05 \\\hline
  \end{tabular}
     \begin{tabular}{|c||c|c|c|c|c|c|}
   \hline
    ~ & \multicolumn{2}{c|}{$n=36$} & \multicolumn{2}{c|}{$n=38$} & \multicolumn{2}{c|}{$n=40$}\\
    ~ & \multicolumn{2}{c|}{$K(\bm{x}) = Y_{32,-24}(\bm{x})$} & \multicolumn{2}{c|}{$K(\bm{x}) = Y_{32,-24}(\bm{x})$} & \multicolumn{2}{c|}{$K(\bm{x}) = Y_{32,-24}(\bm{x})$}\\ \hline
          $m$ & $\|\mathcal{L}_nF-F\|$ & $\|\mathcal{S}_nF-F\|$ & $\|\mathcal{L}_nF-F\|$ & $\|\mathcal{S}_nF-F\|$ & $\|\mathcal{L}_nF-F\|$ & $\|\mathcal{S}_nF-F\|$ \\ \hline
1849 & 0.2092 & 0.0086 & 0.1868 & 0.0031     & 0.1674 & 0.0028 \\\hline
2025 & 0.1622 & 0.0083 & 0.1469 & 6.3438e-04 & 0.1433 & 2.7101e-04 \\\hline
2209 & 0.1327 & 0.0083 & 0.1295 & 5.9311e-04 & 0.1252 & 4.7438e-05 \\\hline
2401 & 0.1180 & 0.0083 & 0.1160 & 5.9286e-04 & 0.1167 & 4.4752e-05 \\\hline
3249 & 0.0736 & 0.0083 & 0.0689 & 5.9286e-04 & 0.0673 & 4.4728e-05 \\\hline
4225 & 0.0432 & 0.0083 & 0.0391 & 5.9286e-04 & 0.0350 & 4.4728e-05 \\\hline
5329 & 0.0174 & 0.0083 & 0.0091 & 5.9286e-04 & 0.0053 & 4.4728e-05 \\\hline
6561 & 0.0083 & 0.0083 & 5.9286e-04 & 5.9286e-04 & 4.4731e-05 & 4.4728e-05 \\\hline
7921 & 0.0083 & 0.0083 & 5.9282e-04 & 5.9286e-04 & 4.4728e-05 & 4.4728e-05 \\\hline
  \end{tabular}
\end{table}

\textbf{Singular functions.} For singular functions, we test three different singular terms. Their forms and the evaluation of modified moments
\begin{equation*}
\beta_{r}:=\beta_{\ell''k''}=\int_{\mathbb{S}^2}K({\bm{x}}) Y_{\ell'',k''}({\bm{x}})\text{d}{\omega(\bm{x})},\quad \ell'' = 0,\ldots,2n,\quad k'' = -\ell'',\ldots,\ell''
\end{equation*}
are elaborated as follows.
\begin{itemize}
       \item Let $K(\bm{x})=|\bm{\xi}-\bm{x}|^{\nu}$, where $\nu>-1$, and $\bm{\xi}$ is an algebraic type singularity if $\nu<0$. Then
              \begin{equation*}
              \beta_{\ell''k''}=2^{\nu+2}\pi\left(-\frac{\nu}{2}\right)_{\ell''}\frac{\Gamma(\frac{\nu+2}{2})}{\Gamma(\ell''+\nu/2+2)}Y_{\ell'',k''}(\bm{\xi}),
              \end{equation*}
              where $\Gamma(\cdot)$ is the Gamma function, and $(\cdot)_n = \Gamma(\cdot+n)/\Gamma(\cdot)$ is the Pochhammer symbol \cite{MR2934227}.
       \item Let $K(\bm{x})=\log{|\bm{\xi}-\bm{x}|}$, where $\bm{\xi}$ is a logarithmic type singularity. Then
              \begin{equation*}
              \beta_{\ell''k''}= \frac{|\mathbb{S}^1|}{2}\left(\int_{-1}^1\log(1-t)P_{\ell''}(t)\text{d}t\right)Y_{\ell'',k''}(\bm{\xi}),
              \end{equation*}
              where $|\mathbb{S}^1|=2\pi$ is the length of the unit circle $\mathbb{S}^1$, and $P_{\ell}$ denote the Legendre polynomials of degree $\ell$ (without normalization).
       \item Let $K(\bm{x})=|\bm{\xi}-\bm{x}|^{\nu_1}|\bm{\xi}+\bm{x}|^{\nu_2}$, where $\nu_1,\nu_2>-1$, and $\bm{\xi}$ and $-\bm{\xi}$ are two algebraic type singularities if $\nu_1,\nu_2<0$. Then
              \begin{equation*}\begin{split}
              \beta_{\ell''k''}= &(-1)^{\ell''}2^{(\nu_1+\nu_2)/2}|\mathbb{S}^1|R_{\ell,3}\\
              &\left(\int_{-1}^1(1-t)^{\nu_1/2}(1+t)^{\nu_2/2}\left[\left(\frac{\text{d}}{\text{d}t}\right)^{\ell''}(1- t^2)^{\ell''}\right]\text{d}t \right)Y_{\ell'',k''}(\bm{\xi}),
              \end{split}\end{equation*}
              where
              \begin{equation*}
              R_{n,s}=\frac{\Gamma(\frac{s-1}{2})}{2^n\Gamma(n+\frac{s-1}{2})}.
              \end{equation*}
\end{itemize}
There results can be found in \cite[Chapter 3]{MR2934227}. In particular, the modified moments of the third term can be evaluated by
\begin{equation*}
\beta_{\ell''k''}= 2^{(\nu_1+\nu_2)/2}|\mathbb{S}^1|\left(\int_{-1}^1(1-t)^{\nu_1/2}(1+t)^{\nu_2/2}P_{\ell''}(t){\text{d}t} \right)Y_{\ell'',k''}(\bm{\xi}),
\end{equation*}
with the aid the Rodrigues' formula
\begin{equation*}
P_{n}(x)={\frac {1}{2^{n}n!}}\left(\frac{\text{d}}{\text{d}x}\right)^n\left[(x^{2}-1)^{n}\right] = (-1)^n{\frac {1}{2^{n}n!}}\left(\frac{\text{d}}{\text{d}x}\right)^n(1-x^{2})^{n}
\end{equation*}
for Legendre polynomials\footnote{It may be unstable to evaluate the integral $\int_{-1}^1 (1-t)^{\nu_1/2}(1 + t)^{\nu_2/2}(\frac{\text{d}}{\text{d}t})^n(1- t^2)^n\text{d}t$ by taking the $n$-th derivative and then evaluating the resulting integral, as the factor accumulated as $n!$ after differentiation may be huge. Thus the error of representing numbers by double-precision floating-point numbers, according to IEEE Standard 754, may be inaccurate.}. For the continuous function $f\in\mathcal{C}(\mathbb{S}^2)$, we consider $f(\bm{x})=f(x,y,z) =e^{x+y+z}$.

For each $K$, we report the $L^1$ errors of classical and efficient hyperinterpolation with $n=2,3,4,\ldots,40$, and $m=(\lceil 1.1n\rceil+1)^2$, $(\lceil 1.2n\rceil+1)^2$, and $(\lceil 1.5n\rceil+1)^2$. The singularity $\bm{\xi}$ in the definitions of $K(\bm{x})$ is set as $\bm{\xi}=[\sqrt{2}/2,\sqrt{2}/2,0]^{\text{T}}$. These errors are plotted in Figure \ref{fig:sphericalsingular}. Unlike the experiments on the singular functions on $[-1,1]$, in which the singularities are always endpoints, all singularities on the sphere are interior. Thus, the numerical integration of spherical singular functions becomes extremely unstable: the actual performance of numerical integration depends on the point distribution around the singularities. This technical issue is also reflected in the approximation of singular functions by numerically integrating the $L^2$ projection coefficients, i.e., the approximation by classical hyperinterpolation. We see from Figure \ref{fig:sphericalsingular} that it seems impossible to predict the actual accuracy of classical hyperinterpolation in the approximation of $F(x,y,z)=K(x,y,z)e^{x+y+z}$, with three kinds of singular $K$ listed above. Indeed, the stability and error bounds of classical hyperinterpolation in \cite{an2022quadrature,sloan1995polynomial} are only valid for the approximation of continuous functions. On the other hand, we see that the actual accuracy of efficient hyperinterpolation is stable and predictable: the point distribution around singularities does not affect the performance of efficient hyperinterpolation, and the approximation error decays as $n$ increases.

\begin{figure}[htbp]
  \centering
  \includegraphics[width=\textwidth]{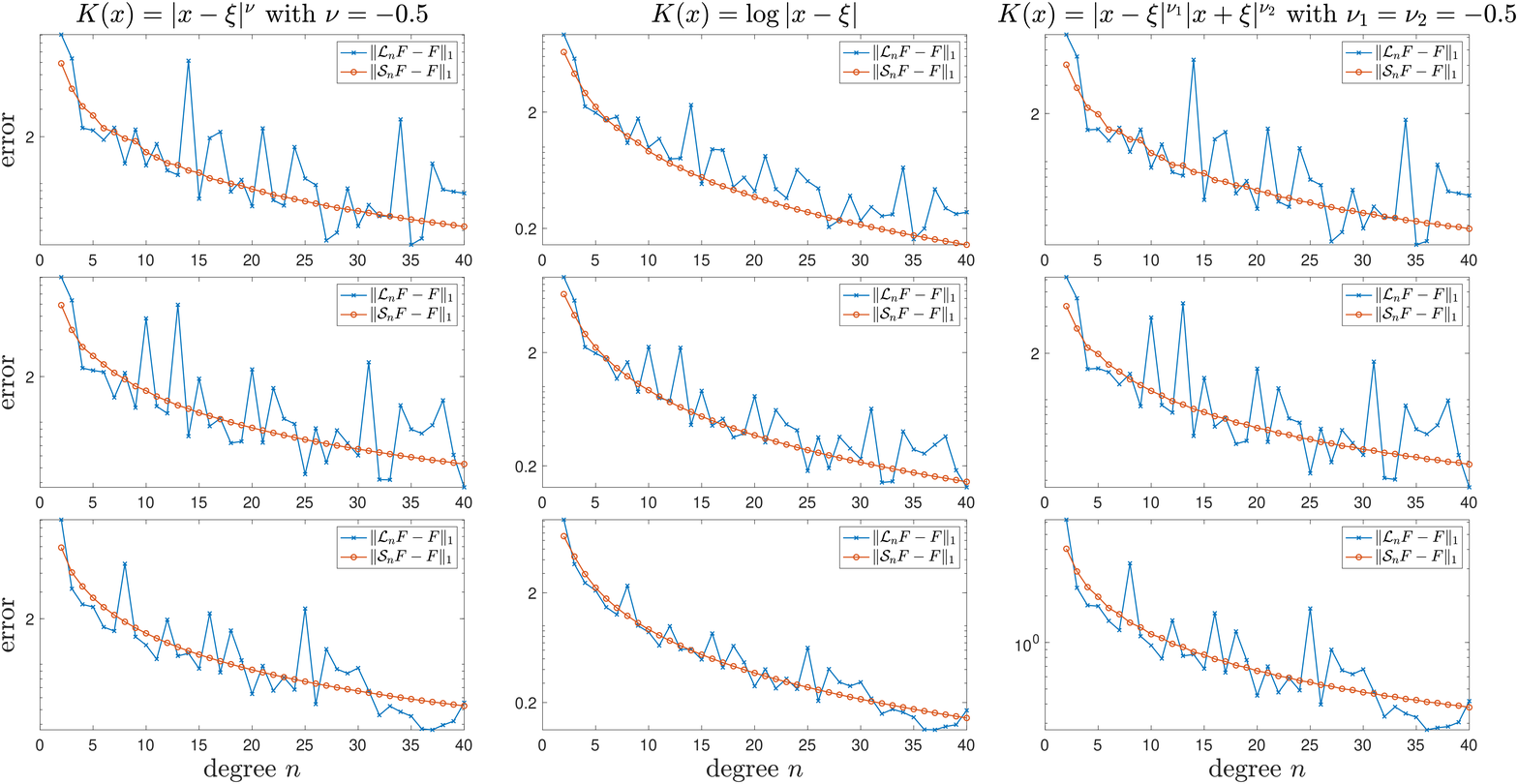}
  \caption{Performance of hyperinterpolation and efficient hyperinterpolation with different $(n,m)$ for the approximation of $F(x,y,z)=K(x,y,z)f(x,y,z)$ with three singular $K$'s and $f(\bm{x})=f(x,y,z)=e^{x+y+z}$. The singularity $\bm{\xi}$ in the definitions of $K(\bm{x})$ is set as $\bm{\xi}=[\sqrt{2}/2,\sqrt{2}/2,0]^{\text{T}}$. From top row to bottom row: $m=(\lceil 1.1n\rceil+1)^2$, $(\lceil 1.2n\rceil+1)^2$, and $(\lceil 1.5n\rceil+1)^2$, respectively.}\label{fig:sphericalsingular}
\end{figure}

\section{Final remarks}
We propose efficient hyperinterpolation to approximate singular and oscillatory functions in the spirit of the product-integration rule. This approximation scheme is new and easy to be implemented. We also obtain error bounds in cases of $K\in L^1(\Omega)$, $ L^2(\Omega)$, and $ L^{\infty}(\Omega)$, respectively. Our theoretical analysis and numerical experiments make it legitimate to apply the proposed scheme to solve problems involving singularity and oscillation functions. On the other hand, efficient hyperinterpolation heavily relies on the accurate or stable evaluation of the modified moments. Thus, much more effort is necessary to understand our scheme's implementation to approximate the function $F=Kf$ with various singular and oscillatory terms $K$.

\section*{Acknowledgement}
We are grateful to Professor Víctor Domínguez of Universidad Pública de Navarra for providing us with the MATLAB subroutine for the stable computation of the modified moments \eqref{equ:1dhighmoments}, which was firstly proposed in \cite{MR2846755}.

\bibliographystyle{siam}
\bibliography{myref}
\clearpage

\end{document}